\documentclass[11pt, a4paper, reqno]{amsart}

\usepackage{amsfonts,amsmath,amsthm}
\usepackage{amssymb}
\usepackage{latexsym}
\usepackage{calrsfs}

\usepackage[english]{babel}

\numberwithin{equation}{section}

\newtheorem{theorem}{Theorem}[section]
\newtheorem{proposition}[theorem]{Proposition}
\newtheorem{lemma}[theorem]{Lemma}
\newtheorem{corollary}[theorem]{Corollary}
\newtheorem{thm}[theorem]{Theorem}
\newtheorem{cor}[theorem]{Corollary}
\newtheorem{prop}[theorem]{Proposition}
\newtheorem{lem}[theorem]{Lemma}

\theoremstyle{definition}

\newtheorem{definition}[theorem]{Definition}
\newtheorem{defn}[theorem]{Definition}
\newtheorem{example}[theorem]{Example}

\theoremstyle{remark}

\newtheorem{remark}[theorem]{Remark}
\newtheorem{rem}[theorem]{Remark}

\newcommand{\eps}{\varepsilon}

\newcommand{\R}{{\mathbb R}}

\newcommand{\Id}{\mathrm{Id}}
\newcommand{\C}{\mathbb{C}}
\newcommand{\K}{\mathbb{K}}

\newcommand{\calS}{\mathcal{S}}

\newcommand{\loglike}[1]{\mathop{\rm #1}\nolimits}

\newcommand{\supp}{\loglike{supp}}

\newcommand{\re}{\loglike{Re}}
\newcommand{\im}{\loglike{Im}}

\newcounter{abc}   
\newcounter{iiiii} 

\newenvironment{aequivalenz}
{\setcounter{iiiii}{0}
\begin{list}%
{{\rm (\roman{iiiii})}}
{\usecounter{iiiii}
\parsep=0pt plus 1pt
\topsep=1pt plus 2pt minus 1pt
\itemsep=1pt plus 2pt minus 1pt
\leftmargin=3\baselineskip \labelsep=.6\baselineskip
\labelwidth=2.4\baselineskip
\rightmargin 0pt}%
}
{\end{list}}

\newenvironment{statements}%
{\setcounter{abc}{0}
\begin{list}%
{{\rm (\alph{abc})}}
{\usecounter{abc}
\parsep=0pt plus 1pt
\topsep=1pt plus 2pt minus 1pt
\itemsep=1pt plus 2pt minus 1pt
\leftmargin=3\baselineskip \labelsep=.6\baselineskip
\labelwidth=2.4\baselineskip
\rightmargin 0pt}%
}
{\end{list}}

{\begin{list}%
{$\bullet$}
{\leftmargin=3\baselineskip
\labelsep=\baselineskip \labelwidth=2.5\baselineskip
\parsep=0pt plus 1pt
\topsep=1pt plus 2pt minus 1pt
\itemsep=1pt plus 2pt minus 1pt
\rightmargin 0pt}%
}
{\end{list}}

\newcommand{\bea}{\begin{eqnarray*}}
\newcommand{\eea}{\end{eqnarray*}}
\newcommand{\beq}{\begin{equation}}
\newcommand{\eeq}{\end{equation}}
\newcommand{\begsta}{\begin{statements}}
\def\endsta{\end{statements}}
\newcommand{\begaeq}{\begin{aequivalenz}}
\def\endaeq{\end{aequivalenz}}

\def\DP{Daugavet property}

\begin{document}

\title{The Daugavet equation for bounded vector valued functions}

\author{Stefan Brach, Enrique A. S\'anchez P\'erez and Dirk Werner}
\date{\today. \textit{File}: \texttt{\jobname.tex}}



\thanks{The authors acknowledge with thanks the
support of the Ministerio de  Econom\'{\i}a y Competitividad
(Spain) under the research project MTM2012-36740-C02-02.}

\address{Department of Mathematics, Freie Universit\"at Berlin,
Arnimallee~6, \qquad {}\linebreak D-14\,195~Berlin, Germany}
\email{brach.stefan@gmail.com}

\address{Instituto
Universitario de Matem\'atica Pura y Aplicada,  Universidad
\linebreak
Polit\'ecnica de Valencia, Camino de Vera s/n, 46022 Valencia, Spain.}
\email{easancpe@mat.upv.es}

\address{Department of Mathematics, Freie Universit\"at Berlin,
Arnimallee~6, \qquad {}\linebreak D-14\,195~Berlin, Germany}
\email{werner@math.fu-berlin.de}

\begin{abstract}
Requirements under which the Daugavet equation and the alternative
Daugavet equation hold for pairs of nonlinear maps between Banach
spaces are analysed. A geometric description is given in
terms of nonlinear slices. 
Some local versions of these properties are also introduced
and studied, as well as  tests for checking if the required
conditions are satisfied in relevant cases. 
\end{abstract}

\maketitle

\thispagestyle{empty}


\section{Introduction}

I.K.~Daugavet \cite{dauga} proved his eponymous equation in 1963 which
establishes the norm identity  
$$
\|\Id + T\| = 1 + \|T\|
$$
for a compact linear operator $T: C[0,1] \to C[0,1]$. This equation
was extended to more general classes of linear operators on various
spaces over the years. Nowadays investigations on this topic build on
the approach of V.~Kadets et al.\ \cite{ams2000} who defined a Banach
space $X$ to have the \DP\ if all rank-$1$ operators on $X$ satisfy
the Daugavet equation. This property can conveniently be characterised
in terms of slices of the unit ball, and it can be shown that on a
space with the \DP\ all weakly compact operators and all operators not
fixing a copy of $\ell_1$ satisfy the Daugavet
equation; see \cite{AAbook}, \cite{ams2000}, \cite{studia2001} or
\cite{Dirk-IrBull}.   

The Daugavet equation has been extended in a number of other ways as
well, replacing 
the identity operator by a more general reference operator called a
Daugavet centre (\cite{bosen}, \cite{boseka}) or replacing the linear
operators $T$ by nonlinear ones (\cite{pol}, \cite{polL1},
\cite{KMMW-Lip}). 
Here we attempt to combine both these ideas. We study the equation
$$
\|\Phi + \Psi\| = \|\Phi\| + \|\Psi\|
$$
where $ \Phi$ and $\Psi$ are bounded maps on the unit ball of some
Banach space $X$ having values in some (possibly different) Banach
space $Y$ and $\Psi$ is in some sense small with respect to $\Phi$,
the norm being the sup norm. Also, the so-called alternative Daugavet
equation 
$$
\max_{|\omega|=1} \|\Phi + \omega\Psi\| = \|\Phi\| + \|\Psi\|
$$
will be considered. We are going to investigate these equations by
means of suitable modifications of the notion of slice continuity
introduced in \cite{SPWFuncAp};  
cf.\ Definition~\ref{def:natural set of slices} below.
We also rely on some techniques from \cite{pol} and \cite{polL1}.

The paper is organised as follows. After the preliminary Section~2,
we study the $\Phi$-Daugavet   equation  in the third section, giving
complete characterisations using the notion of strong slice continuity
introduced below. Likewise, we introduce weak slice continuity in
order to deal with the alternative Daugavet equation in Section~4.  
Finally, Section~5 is devoted to some technical local versions of
these Daugavet type properties which  are obtained by considering
suitable subsets of the ones appearing in the definitions studied
before. Some tests that guarantee that the requirements in our main
theorems are satisfied  are also presented. In particular, examples
show their usefulness, especially for the cases of $C(K)$-spaces and
$L^1(\mu)$-spaces.

Let us introduce some fundamental definitions and notation. We will
write $\mathbb T$ for the set of scalars of modulus~$1$; the field of
scalars can be $\K=\R$ or $\K=\C$.  We write $\re \omega$ for the real
part, $\im \omega$ for the imaginary part  and $\overline{\omega}$ for
the complex conjugate of~$\omega$. For a Banach space $X$, $B_X$ is
its closed unit ball and $S_X$ its unit sphere, and we will denote by
$X^*$ its dual space.   
If $L$ is a Banach lattice, we use the symbol $L^+$ to denote the
positive cone, and $B_{L^+}$ for the set $B_L \cap L^+$.
$L(X,Y)$ denotes the space of continuous linear operators from $X$ to
$Y$. 

 For a bounded mapping $\Phi:B_X \to Y$, we define its  norm  to be
 the sup norm, i.e.,  
$$
\|\Phi\| := \sup_{x \in B_X} \|\Phi(x)\|;
$$
the space of all such mappings is denoted by $\ell_\infty(B_X, Y)$. In
the scalar case an element of $\ell_\infty(B_X)$ is typically denoted
by~$x'$. The symbol $x'\otimes y$ stands for the mapping $x\mapsto 
x'(x)y$. 

Our main characterizations are given in terms of slices. A slice 
$S(x^*, \varepsilon)$ of  $B_X$ 
defined by a norm one element $x^* \in X^*$ and an $\varepsilon>0$ 
is defined by
$$
S(x^*, \varepsilon)= \{ x \in B_X :  \re x^*(x) \ge 1- \varepsilon \}.
$$
When a nonlinear scalar-valued   function is considered, the same
definition makes sense; if $p:X \to \mathbb K$ is a function with norm
$\le1$, we write 
$$
S(p, \varepsilon)= \{ x \in B_X : \re p(x) \ge 1-  \varepsilon \}.
$$
Note that in this case it may happen that 
$S(p,\eps)=\emptyset$. 


\section{\label{chap:Preliminary-Generalities}Preliminaries}

In this section, we prove fundamental characterisations of the Daugavet
and the alternative Daugavet equation. The theorems in this section
are  adapted from
results in \cite{pol} and \cite{SPWFuncAp}.

\begin{defn}
Let $X,Y$ be Banach spaces and let $\Phi \in\ell_{\infty}(B_{X},Y)$.
We say that $\Psi\in\ell_{\infty}(B_{X},Y)$ satisfies the
\textit{$\Phi $-Daugavet equation} if the norm equality
\begin{equation}
\tag{$\Phi $-DE}
\|\Phi +\Psi\|=\|\Phi \|+\|\Psi\|
\end{equation}
holds. If $\Phi $ is the restriction of the identity to $B_X$, we call the above
equation the \textit{Daugavet equation} (DE).
\end{defn}

To connect the Daugavet equation to a set $V\subset\ell_{\infty}(B_{X},Y)$,
we establish the following terminology.

\begin{defn}
Let $X,Y$ be Banach spaces and let $\Phi \in\ell_{\infty}(B_{X},Y)$.
\begin{enumerate}
\item[(1)]
$Y$ has the \textit{$\Phi $-Daugavet property with respect to
$V\subset\ell_{\infty}(B_{X},Y)$}
if ($\Phi $-DE) is satisfied by all $\Psi\in V$.
\item[(2)]
$Y$ has the \textit{$\Phi $-Daugavet property} if
$\|\Phi +R\|=\|\Phi \|+\|R \|$
for all $R \in L(X,Y)$ with one-dimensional range.
\item[(3)]
$Y$ has the \textit{Daugavet property} if (2) holds for $X=Y$ and
$\Phi =\Id$.
\end{enumerate}
\end{defn}

The following lemma (see e.g.\ 
\cite[Lemma~11.4]{AAbook} or
\cite{Dirk-IrBull} for a proof)
frequently simplifies proofs concerning the Daugavet
equation, because we only need to consider maps of norm~$1$. We will
often make use of the lemma without explicitly mentioning it.

\begin{lem}
\label{lem:betaT}
Two
vectors $u$ and $v$ in a normed space satisfy $\|u+v\|=\|u\|+\|v\|$
if and only if $\|\alpha u+\beta v\|=\alpha\|u\|+\beta\|v\|$ holds
for all $\alpha,\beta\geq0$. In particular, $\Psi$ satisfies
$\mathrm{(}\Phi $-$\mathrm{DE)}$
if and only if $\alpha\Psi$ satisfies $\mathrm{(}\beta \Phi $-$\mathrm{DE)}$
for all $\alpha,\beta\geq0$.
\end{lem}

To prove the first theorem of this section, we need the following
lemma.

\begin{lem}
\label{lem:x in S(x',epsilon) implies |1-x'x| leq 2epsilon stefan}
Let
$X$ be a Banach space and assume $x'\in\ell_{\infty}(B_{X})$
with $\|x'\|\leq1$. Let $0\leq\varepsilon\leq1$ and $x\in B_{X}$.
Then $\re x'(x)\geq1-\varepsilon$ implies
$|1-x'(x)|\leq\sqrt{2\varepsilon}$.
\end{lem}

\begin{proof}
First note that
\[
1  \geq  |x'(x)|^{2}
  =  \left(\im x'(x)\right)^{2}+\left(\re x'(x)\right)^{2}
  \geq  \left(\im x'(x)\right)^{2}+(1-\varepsilon)^{2}.
\]
Hence
\[
\left(\im x'(x)\right)^{2}  \leq  1-(1-\varepsilon)^{2}
  =  2\varepsilon-\varepsilon^{2}.
\]
Since $\re x'(x)\geq1-\varepsilon$ and $|x'(x)|\leq1$,
we know that $0\leq1-\re x'(x)\leq\varepsilon$. Thus
\begin{eqnarray*}
|1-x'(x)|^{2} & = & |1-\re x'(x)-i\im x'(x)|^{2}\\
 & = & (1-\re x'(x))^{2}+(\im x'(x))^{2}\\
 & \leq & \varepsilon^{2}+2\varepsilon-\varepsilon^{2}\\
 & = & 2\varepsilon,
\end{eqnarray*}
i.e., $\left|1-x'(x)\right|\leq\sqrt{2\varepsilon}$.
\end{proof}

\begin{thm}
\label{thm:generalization daugavet center one map}
Let $X,Y$ be Banach
spaces. Let $\Phi \in\ell_{\infty}(B_{X},Y)$ and consider a norm one
map $x'\in\ell_{\infty}(B_{X})$ and $y\in Y\backslash\{0\}$.
Then the following are equivalent:
\begin{enumerate}
\item[(1)]
$\|\Phi +x'\otimes y\|=\|\Phi \|+\|y\|$.
\item[(2)]
For every $\varepsilon>0$ there are $x\in B_{X}$ and
$\omega\in\mathbb{T}$ such that
\[
\re \omega x'(x)\geq1-\varepsilon\quad\mbox{and}\quad\left\Vert \omega
  \Phi (x)+\frac{y}{\|y\|}\right\Vert \geq\|\Phi \|+1-\varepsilon.
\]
\end{enumerate}
\end{thm}

\begin{proof}
(1) $\Rightarrow$ (2):
By Lemma~\ref{lem:betaT},
we can assume $y\in S_{Y}$. Hence there is an element $x\in B_{X}$
such that
\begin{eqnarray*}
\|\Phi \|+1-\frac{\varepsilon}{2} & \leq & \|\Phi (x)+x'(x)y\|\\
 & \leq & \|\Phi (x)\|+|x'(x)|\|y\|\\
 & \leq & \|\Phi \|+|x'(x)|.
\end{eqnarray*}
Thus $|x'(x)|\geq1-\frac{\varepsilon}{2}$. Writing
$\omega=|x'(x)|/x'(x)\in\mathbb{T}$ we have
\[
\re \omega x'(x)=|x'(x)|\geq1-\varepsilon.
\]
Moreover
\begin{eqnarray*}
\|\Phi \|+1-\frac{\varepsilon}{2} & \leq & \|\Phi (x)+x'(x)y\|\\
 & = & \|\omega \Phi (x)+\omega x'(x)y\|\\
 & \leq & \|\omega \Phi (x)+y\|+\|\omega x'(x)y-y\|\\
 & = & \|\omega \Phi (x)+y\|+|\omega x'(x)-1|\|y\|\\
 & = & \|\omega \Phi (x)+y\|+|\mbox{ }|x'(x)|-1|\\
 & \leq & \|\omega \Phi (x)+y\|+\frac{\varepsilon}{2},
\end{eqnarray*}
and (2) follows.

(2) $\Rightarrow$ (1):
Again, by Lemma~\ref{lem:betaT},
it suffices to consider the case $\|y\|=1$. Let $\varepsilon>0$
and take $x\in B_{X}$ and $\omega\in\mathbb{T}$ such that
\[
\re \omega x'(x)\geq1-\varepsilon\quad\mbox{and}\quad\|\omega \Phi
(x)+y\|\geq\|\Phi \|+1-\varepsilon.
\]
Thus
\begin{eqnarray*}
\|\Phi \|+1-\varepsilon & \leq & \|\omega \Phi (x)+y\|\\
 & = & \|\Phi (x)+\overline{\omega}y\|\\
 & \leq & \|\Phi (x)+x'(x)y\|+\|\overline{\omega}y-x'(x)y\|\\
 & = & \|\Phi (x)+x'(x)y\|+\|y-\omega x'(x)y\|\\
 & = & \|\Phi (x)+x'(x)y\|+|1-\omega x'(x)|\\
 & \leq & \|\Phi (x)+x'(x)y\|+2\varepsilon,
\end{eqnarray*}
where the last inequality is due to Lemma~\ref{lem:x in S(x',epsilon)
  implies |1-x'x| leq 2epsilon stefan}.
Since $\varepsilon$ was arbitrary, (1) holds.
\end{proof}

Next we present analogous results in the setting of the alternative
Daugavet equation.

\begin{defn}
Let $X,Y$ be Banach spaces and $\Phi ,\Psi\in\ell_{\infty}(B_{X},Y)$.
We say that $\Psi$ satisfies the
\textit{alternative $\Phi $-Daugavet equation}
if
\begin{equation}
\tag{$\Phi $-ADE}
\max_{|\omega|=1}\|\Phi +\omega\Psi\|=\|\Phi \|+\|\Psi\|
\end{equation}
is true. In the case where $\Phi $ is the identity, we refer to the above
equation simply as the alternative Daugavet equation (ADE).
\end{defn}

We will also make use of the following definitions regarding a set
$V\subset\ell_{\infty}(B_{X},Y)$.

\begin{defn}
Let $X,Y$ be Banach spaces and let $\Phi \in\ell_{\infty}(B_{X},Y)$.
\begsta
\item
$Y$ has the \textit{alternative $\Phi $-Daugavet property with respect
to $V\subset\ell_{\infty}(B_{X},Y)$} if ($\Phi $-ADE) is satisfied by
all $\Psi\in V$.
\item
$Y$ has the \textit{alternative $\Phi $-Daugavet property} if
$\max_{|\omega|=1}\|\Phi +\omega R \|=\|\Phi \|+\|R\|$
for all $R \in L(X,Y)$ with one-dimensional range.
\item
$Y$ has the \textit{alternative Daugavet property} if it has the alternative
$\Id$-\DP.
\endsta
\end{defn}

Now that the notation is fixed, let us look at how the Daugavet and
the alternative Daugavet equation are interrelated.

\begin{rem}
\label{rem:remark ade}
Let $X,Y$ be Banach spaces and $\Phi, \Psi \in\ell_{\infty}(B_{X},Y)$.
\begin{enumerate}
\item[(1)]
$\Psi$ satisfies ($\Phi $-ADE) if and only if there exists
$\omega\in\mathbb{T}$ such that $\omega\Psi$ fulfills ($\Phi $-DE).
\item[(2)]
($\Phi $-DE) implies ($\Phi $-ADE), but, in general, the converse
is not true. For example, $-\mathrm{Id}$ always satisfies (ADE),
but never (DE).
\item[(3)]
$\Psi$ satisfies ($\Phi $-$A\mathrm{DE}$) if and only if
$\alpha\Psi$ satisfies ($\beta \Phi $-$A\mathrm{DE}$) for every
$\alpha,\beta\geq0$.
This is a consequence of (1) and Lemma~\ref{lem:betaT}.
\end{enumerate}
\end{rem}

\begin{thm}
\label{thm:generalization alternative daugavet center one map}
Let
$X,Y$ be Banach spaces. Let $\Phi \in\ell_{\infty}(B_{X},Y)$ and consider
a norm one map $x'\in\ell_{\infty}(B_{X})$ and $y\in Y\backslash\{0\}$.
Then the following are equivalent:
\begin{enumerate}
\item[(1)]
$\max_{|\omega|=1}\|\Phi +\omega x'\otimes y\|=\|\Phi \|+\|y\|$.
\item[(2)]
For every $\varepsilon>0$ there exist $\omega_{1},\omega_{2}\in\mathbb{T}$
and $x\in B_{X}$ such that
\[
\re \omega_{1}x'(x)\geq1-\varepsilon\quad\mbox{and}\quad\left\Vert
  \omega_{2}\Phi (x)+\frac{y}{\|y\|}\right\Vert \geq\|\Phi
\|+1-\varepsilon.
\]
\item[(3)]
For every $\varepsilon>0$ there exist $\omega\in\mathbb{T}$
and $x\in B_{X}$ such that
\[
\left|x'(x)\right|\geq1-\varepsilon\quad\mbox{and}\quad\left\Vert \omega
  \Phi (x)+\frac{y}{\|y\|}\right\Vert \geq\|\Phi \|+1-\varepsilon.
\]
\end{enumerate}
\end{thm}

\begin{proof}
(1) $\Rightarrow$ (2):
By Remark~\ref{rem:remark ade}(3), we can
assume $\|y\|=1$. According to (1), there exists $\omega\in\mathbb{T}$
such that $\|\Phi +\omega x'\otimes y\|=\|\Phi \|+1$. Thus, for given
$\varepsilon>0$, Theorem~\ref{thm:generalization daugavet center one map}
yields $x\in B_{X}$ and $\omega_{2}\in\mathbb{T}$ such that
\[
\re \omega_{2}\omega
x'(x)\geq1-\varepsilon\quad\mbox{and}\quad\left\Vert \omega_{2}\Phi
  (x)+y\right\Vert \geq\|\Phi \|+1-\varepsilon.
\]
Defining $\omega_{1}=\omega_{2}\omega$, (2) follows.

(2) $\Rightarrow$ (3):
If $\re \omega_{1}x'(x)\geq1-\varepsilon$,
then
\[
1-\varepsilon\leq\re \omega_{1}x'(x)\leq|x'(x)|.
\]

(3) $\Rightarrow$ (1):
It suffices to consider the case $\|y\|=1$.
For given $\varepsilon>0$, take $\omega\in\mathbb{T}$ and $x\in B_{X}$
such that
\[
|x'(x)|\geq1-\varepsilon\quad\mbox{and}\quad\left\Vert \omega \Phi
 (x)+y\right\Vert \geq\|\Phi \|+1-\varepsilon.
\]
Denote $\omega_{1}=|x'(x)|/x'(x)$ and $\omega_{2}=\overline{\omega}\omega_{1}$.
Thus
\begin{eqnarray*}
\|\Phi (x)+\omega_{2}x'(x)y\| & = & \|\Phi
 (x)+\overline{\omega}\omega_{1}x'(x)y\|\\
 & = & \|\omega \Phi (x)+\omega_{1}x'(x)y\|\\
 & \geq & \|\omega \Phi (x)+y\|-\|y-\omega_{1}x'(x)y\|\\
 & = & \|\omega \Phi (x)+y\|-|1-|x'(x)|\,|\\
 & \geq & \|\Phi \|+1-2\varepsilon,
\end{eqnarray*}
and we are done since $\varepsilon>0$ was arbitrary.
\end{proof}


\section{\label{chap:Strong-Slice-Continuity}Strong slice continuity}

In  \cite{SPWFuncAp}
the notion of slice continuity was introduced
to study when the
Daugavet equation holds for a couple of maps $\Phi$ and $\Psi$ between
Banach spaces, i.e., when
\[
\|\Phi+\Psi\|=\|\Phi\|+\|\Psi\|.
\]

The functions taken into account were either linear or bilinear bounded
maps. In this section, we will extend some of the results from \cite{SPWFuncAp}
to the case of bounded nonlinear functions.

The following definition is from \cite{SPWFuncAp}.

\begin{defn}
\label{def:natural set of slices}
Let $X,Y$ be Banach spaces and
$\Phi\in\ell_{\infty}(B_{X},Y)$.
\begsta
\item
If $y^{*}\in Y^{*}$ with $y^{*}\Phi\neq0$, we define $\Phi_{y^{*}}:B_{X}\to\mathbb{K}$
by
\[
\Phi_{y^{*}}(x)=\frac{1}{\|y^{*}\Phi\|}y^{*}\Phi(x).
\]
\item
The \textit{natural set of slices} defined by $\Phi$ is given by
\[
\calS_{\Phi}=\{S(\Phi_{y^{*}},\varepsilon): 0<\varepsilon<1,\, y^{*}\in Y^{*},\,
y^{*}\Phi\neq0\}.
\]
\item
We write $\calS_{\Psi}\leq \calS_{\Phi}$ if for every
$S(\Psi_{z^{*}},\varepsilon)\in \calS_{\Psi}$ 
there is $S(\Phi_{y^{*}},\mu)\in \calS_{\Phi}$ with
\[
S(\Phi_{y^{*}},\mu)\subset S(\Psi_{z^{*}},\varepsilon).
\]
In this instance we say that $\Psi$ is \textit{slice continuous} with respect
to $\Phi$.
\endsta
\end{defn}

Now we are ready to introduce the concept of strong slice continuity
for bounded nonlinear maps.

\begin{defn}
\label{def:slice continuity}
Let $X,Y$ be Banach spaces and $\Phi,\Psi\in\ell_{\infty}(B_{X},Y)$.
We use the symbol $\calS_{\Psi}<\calS_{\Phi}$ if for every slice
$S(\Psi_{z^{*}},\varepsilon)\in \calS_{\Psi}$
there is another slice $S(\Phi_{y^{*}},\mu)\in \calS_{\Phi}$ such that
\[
S(\omega \Phi_{y^{*}},\mu)\subset S(\omega
\Psi_{z^{*}},\varepsilon)\quad\mbox{for all }\omega\in\mathbb{T}.
\]
In this case we say that $\Psi$ is \textit{strongly slice continuous}
with respect to $\Phi$.
\end{defn}

Note that the above and similar definitions carry over to bounded
functions from $X$ to $Y$ by considering the respective restrictions
to $B_{X}$.

It is clear that strong slice continuity implies slice continuity.
The following remark shows that in the case of multilinear maps, the
two concepts coincide.

\begin{rem}
Let $X_{1},\dots,X_{n},Z$ be Banach spaces and
$A,B:X_{1}\times\cdots\times X_{n}\to Z$
bounded multilinear maps. Then $\calS_{A}<\calS_{B}$ if and only if
$\calS_{A}\leq \calS_{B}$.
\end{rem}

\begin{proof}
We only need to verify that slice continuity implies strong slice
continuity. To this end, let $S(A_{x^{*}},\varepsilon)\in \calS_{A}$
be given. Since $\calS_{A}\leq \calS_{B}$, we can find
$S(B_{y^{*}},\mu)\in \calS_{B}$ 
with $S(B_{y^{*}},\mu)\subset S(A_{x^{*}},\varepsilon)$. For a given
$\omega\in\mathbb{T}$ and $(x_{1},\dots,x_{n})\in S(\omega B_{y^{*}},\mu)$,
we have
\[
1-\mu\leq\re \omega\frac{y^{*}B(x_{1},\dots,x_{n})}{\|y^{*}B\|}=\re
\frac{y^{*}B(\omega x_{1},\dots,x_{n})}{\|y^{*}B\|},
\]
i.e., $(\omega x_{1},x_2, \dots,x_{n})\in S(B_{y^{*}},\mu)$. This ensures
$(\omega x_{1},x_2, \dots,x_{n})\in S(A_{x^{*}},\varepsilon)$, and the
multilinearity of $A$ leads to $(x_{1},x_2, \dots,x_{n})\in S(\omega
A_{x^{*}},\varepsilon)$.
\end{proof}

The canonical example of when the relation $\calS_{\Psi}<\calS_{\Phi}$ holds is
given by the case where $\Psi$ is the concatenation of a map $\Phi$ and
a bounded linear operator.

\begin{example}
\label{exa:strong slice continuity concatenation}
Let $X,Y$ be Banach
spaces. Consider $\Phi\in\ell_{\infty}(B_{X},Y)$ and a bounded linear
operator $P:Y\to Y$. Denote $\Psi=P\circ \Phi$. Then $S_{\Psi}<S_{\Phi}$.
\end{example}

\begin{proof}
Let $S(\Psi_{y^{*}},\varepsilon)$ be a slice in $\calS_{\Psi}$. First note
that since $y^{*}\Psi\neq0$, we also have $(y^{*}P)\Phi=y^{*}\Psi\neq0$,
and thus $S(\Phi_{y^{*}P},\varepsilon)\in \calS_{\Phi}$. Take
$\omega\in\mathbb{T}$ 
and $x\in S(\omega \Phi_{y^{*}P},\varepsilon)$, i.e.,
\[
\re \omega\frac{(y^{*}P)\Phi}{\|(y^{*}P)\Phi\|}(x)\geq1-\varepsilon.
\]
By construction,
\[
\re \omega\frac{y^{*}\Psi}{\|y^{*}\Psi\|}(x)=\re
\omega\frac{(y^{*}P)\Phi}{\|(y^{*}P)\Phi\|}(x)\geq1-\varepsilon, 
\]
and therefore $x\in S(\omega \Psi_{y^{*}},\varepsilon)$.
\end{proof}

The next example shows that there are bounded maps $\Phi,\Psi$ with
$\calS_{\Psi}<\calS_{\Phi}$, 
but $\Psi\neq P\circ \Phi$ for any bounded linear operator $P$.

\begin{example}
\label{exa:2.5 slice continuity-1-1}
Let $C[0,1]$ denote the Banach
space of continuous functions from $[0,1]$ to $\mathbb{K}$. Let
$\Phi:C[0,1]\oplus_{1}\mathbb{K}\to C[0,1]$, $\Phi(f,\alpha)=f$, and
$\Psi:C[0,1]\oplus_{1}\mathbb{K}\to C[0,1]$,
$\Psi(f,\alpha)=f+\alpha^{2}\mathbf{1}$, 
where $\mathbf{1}$ stands for the constant one function and $\oplus_{1}$
denotes the direct sum with the $1$-norm. Then $\Psi$ and $\Phi$ have
norm one. The kernel of $\Phi$ is not contained in the kernel of $\Psi$,
since $\Phi(0,1)=0$, but $\Psi(0,1)\neq0$. Thus we do not have
$\Psi=P\circ \Phi$ 
for any bounded linear operator $P$. But the slice condition
$\calS_{\Psi}<\calS_{\Phi}$ 
holds. First note that for any $x^{*}\in C[0,1]^{*}\backslash\{0\}$,
we have $\|x^{*}\Phi\|=\|x^{*}\Psi\|=\|x^{*}\|\neq0$. Consider some
$x^{*}\in C[0,1]^{*}$
with $\|x^{*}\|=1$, and let $0<\varepsilon<1$. We claim $S\left(\omega
  x^{*}\Phi,\frac{\varepsilon}{2}\right)\subset S\left(\omega
  x^{*}\Psi,\varepsilon\right)$
for all $\omega\in\mathbb{T}$. To prove this, assume $(f,\alpha)\in
S\left(\omega x^{*}\Phi,\frac{\varepsilon}{2}\right)$,
i.e., $\re \omega x^{*}(f)\geq1-\frac{\varepsilon}{2}$. In
particular, $\|f\|\geq1-\frac{\varepsilon}{2}$, and therefore
$|\alpha|\leq\varepsilon/2$.
Hence
\begin{eqnarray*}
\re \omega x^{*}\Psi(f,\alpha) & = & \re \omega x^{*}(f+\alpha^{2}\mathbf{1})\\
 & = & \re \omega x^{*}(f)+\re \omega x^{*}(\alpha^{2}\mathbf{1})\\
 & \geq & 1-\varepsilon.
\end{eqnarray*}
\end{example}

Consider now a closed subspace $U$ of a normed space $X$. Then $q:X\to X/U$,
$q(x)=x+U$, sends the open unit ball $U_{X}$ onto $U_{X/U}$. This
motivates the following definition.

\begin{defn}
\label{def:quotient map}
Let $X,Y$ be Banach spaces. We call
$\Phi\in\ell_{\infty}(B_{X},Y)$
a \textit{quotient map} if $\Phi$ is continuous and $\Phi(U_{X})=U_{Y}$.
\end{defn}

Given $\Phi\in \ell_\infty(B_X,Y)$ we set 
$$
Y^*\Phi \cdot Y = \{y^*\Phi \otimes
y : y^*\in Y^*, \ y\in Y\}.
$$

\begin{lem}
\label{lem:characterization of when T is a quotient map and Y has the
  daugavet property}
Let
$X,Y$ be Banach spaces and assume $\Phi\in\ell_{\infty}(B_{X},Y)$ is
a quotient map. Then the following are equivalent:
\begin{enumerate}
\item[(1)]
$Y$ has the Daugavet property.
\item[(2)]
$Y$ has the $\Phi$-Daugavet property with respect to $Y^{*}\Phi\cdot Y$.
\end{enumerate}
\end{lem}

\begin{proof}
This is a consequence of the assumptions that $\Phi$ is continuous and
$\Phi(U_{X})=U_{Y}$.
\end{proof}

\begin{prop}
\label{prop:norm of T plus y star R ot y equals 2}
Let $X,Y$ be Banach
spaces and assume $Y$ has the Daugavet property. Consider
$\Psi,\Phi \in\ell_{\infty}(B_{X},Y)$
such that $\Phi $ is a quotient map and $\|\Psi\|=1$. Then
$\calS_{\Psi}<\calS_{\Phi }$ 
implies that for every $y\in S_{Y}$ and $y^{*}\in Y^{*}$ with $y^{*}\Psi\neq0$
\[
\|\Phi +\Psi_{y^{*}}\otimes y\|=2.
\]
\end{prop}

\begin{proof}
By Theorem~\ref{thm:generalization daugavet center one map}, it suffices
to show that for every $\varepsilon>0$ there are $\omega\in\mathbb{T}$
and $x\in S(\omega \Psi_{y^{*}},\varepsilon)$ such that
\[
\left\Vert \omega \Phi (x)+y\right\Vert \geq2-\varepsilon.
\]

Thus, let $\varepsilon>0$ be given. Since $\calS_{\Psi}<\calS_{\Phi }$, we can
find a slice $S(\Phi _{z^{*}},\mu)\in \calS_{\Phi }$ with $\mu\leq\varepsilon$
such that $S(\lambda \Phi _{z^{*}},\mu)\subset S(\lambda
\Psi_{y^{*}},\varepsilon)$ 
for all $\lambda\in\mathbb{T}$. According to Lemma
\ref{lem:characterization of when T is a quotient map and Y has the
  daugavet property},
$\|\Phi +\Phi _{z^{*}}\otimes y\|=2$, therefore Theorem
\ref{thm:generalization daugavet center one map}
gives $\omega\in\mathbb{T}$ and $x\in S(\omega \Phi _{z^{*}},\mu)$ satisfying
\[
\|\omega \Phi (x)+y\|\geq2-\mu\geq2-\varepsilon.
\]
By construction, $S(\omega \Phi _{z^{*}},\mu)\subset S(\omega
\Psi_{y^{*}},\varepsilon)$,
and the proof is complete.
\end{proof}

\begin{rem}
\label{rem:strong slice continuity cant be removed ex 1}
The above
proposition is false if the condition $\calS_{\Psi}<\calS_{\Phi }$ is
removed. To 
see this, consider bounded linear operators
$\Phi ,\Psi:L_{1}[0,1]\oplus_{1}L_{1}[1,2]\to L_{1}[0,1]$
given by $\Phi ((f,g))=f$ and $\Psi((f,g))=(\int_{1}^{2}g\, dx)\cdot\mathbf{1}$,
where $(f,g)\in L_{1}[0,1]\oplus_{1}L_{1}[1,2]$; recall that $L_{1}[0,1]$
has the Daugavet property. Clearly, $\Phi $ is a quotient map and $\|\Psi\|=1$.
But, if $y=\mathbf{1}\in L_{1}[0,1]$ and $y^{*}=\mathbf{1}\in L_{\infty}[0,1]$,
then $\|\Phi +y^{*}\Psi\otimes y\|\leq1$.
\end{rem}

We shall now deal with weakly compact maps.
Let us start  by recalling the definition of a (nonlinear)
weakly compact map.

\begin{defn}
Let $X,Y$ be Banach spaces. A function $\Psi\in\ell_{\infty}(B_{X},Y)$
is called \textit{weakly compact} if the weak closure of $\Psi(B_{X})$ is
a weakly compact set.
\end{defn}

Let us now  prove the main result of this section, namely
Theorem~\ref{thm:daugavet property weakly compact main theorem}.

\begin{thm}
\label{thm:daugavet property weakly compact main theorem}
Let $X,Y$
be Banach spaces and let $\Phi ,\Upsilon,\Psi\in\ell_{\infty}(B_{X},Y)$ with
$\|\Phi \|=\|\Upsilon\|=\|\Psi\|=1$. Assume $Y$ has the $\Phi
$-Daugavet property 
with respect to $Y^{*} \Upsilon \cdot Y$. Then, if
$\calS_{\Psi}<\calS_{\Upsilon}$ and $\Psi$ 
is weakly compact,
\[
\|\Phi +\Psi\|=2.
\]
\end{thm}

\begin{proof}
Since
the set $K=\overline{\mathrm{co}}(\mathbb{T}\Psi(B_{X}))$ is weakly
compact by Krein's theorem, we can conclude  
that $K$ coincides with the closed convex hull of its strongly exposed
points (\cite{Bou76}, \cite[Cor.~5.18]{BL1}). 
Therefore, given $\varepsilon>0$, we may take a strongly
exposed point $y_{0}\in K$ with $\|y_{0}\|>1-\varepsilon$. Because
$y_{0}$ is a strongly exposed point, there are $z^{*}\in Y^{*}$ and
$\eta>0$ such that the set
\[
\left\{ y\in K: \re z^{*}(y)\geq\re z^{*}(y_{0})-\eta\right\}
\]
has diameter less than $\varepsilon$ and $\re z^{*}(y_{0})>\re z^{*}(y)$
for all $y\in K\backslash\{y_{0}\}$. After defining
$y_{0}^{*}=z^{*}/\re z^{*}(y_{0})$
and $\delta=\min\{\frac{\varepsilon}{2},\eta/\re z^{*}(y_{0})\}$,
we have found a slice
\[
S=\{y\in K: \re y_{0}^{*}(y)\geq1-\delta\}
\]
containing $y_{0}$ and having diameter less than $\varepsilon$.
In particular,
\[
y\in K,\ \re
y_{0}^{*}(y)\geq1-\delta\quad\Rightarrow\quad\|y-y_{0}\|<\varepsilon.
\]
Also note that since $K$ is balanced,
\[
\sup_{y\in K}\re y_{0}^{*}(y)=\sup_{y\in K}|y_{0}^{*}(y)|=1.
\]
Denote $\psi:=y_{0}^{*}\circ\Psi$. We have
\[
\|\psi\|=\sup_{x\in B_{X}}|y_{0}^{*}(\Psi(x))|=\sup_{y\in K}|y_{0}^{*}(y)|=1,
\]
hence $S(\psi,\delta)\in \calS_{\Psi}$. On account of
$\calS_{\Psi}<\calS_{\Upsilon}$, 
there are $\mu\leq\delta$ and  $S(\Upsilon_{z^{*}},\mu)\in
\calS_{\Upsilon}$ such that 
\[
S(\lambda \Upsilon_{z^{*}},\mu)\subset S(\lambda\psi,\delta)\quad\mbox{for
  all }\lambda\in\mathbb{T}.
\]
Since by assumption $\|\Phi +\Upsilon_{z^{*}}\otimes
y_{0}\|=1+\|y_{0}\|$, Theorem 
\ref{thm:generalization daugavet center one map} yields $\omega\in\mathbb{T}$
and $x\in S(\omega \Upsilon_{z^{*}},\mu)$ so that
\[
\left\Vert \omega \Phi (x)+\frac{y_{0}}{\|y_{0}\|}\right\Vert
\geq2-\mu\geq2-\varepsilon.
\]
By construction, $x\in S(\omega \Upsilon_{z^{*}},\mu)\subset
S(\omega\psi,\delta)$, 
and therefore
\[
\re y_{0}^{*}(\omega\Psi(x))=\re \omega\psi(x)\geq1-\delta,
\]
so the fact that $\omega\Psi(x)\in K$ gives
$\|\omega\Psi(x)-y_{0}\|<\varepsilon$.

We calculate
\begin{eqnarray*}
\|y_{0}+\omega \Phi (x)\| & \geq & \left\Vert \omega
 \Phi (x)+\frac{y_{0}}{\|y_{0}\|}\right\Vert -\left\Vert
 y_{0}-\frac{y_{0}}{\|y_{0}\|}\right\Vert \\
 & = & \left\Vert \omega \Phi (x)+\frac{y_{0}}{\|y_{0}\|}\right\Vert
 -|\,\|y_{0}\|-1|\\
 & \geq & 2-2\varepsilon.
\end{eqnarray*}
Finally,
\begin{eqnarray*}
\|\Phi +\Psi\| & \geq & \|\Phi (x)+\Psi(x)\|\\
 & = & \|\omega \Phi (x)+\omega\Psi(x)\|\\
 & \geq & \|\omega \Phi (x)+y_{0}\|-\|\omega\Psi(x)-y_{0}\|\\
 & \geq & 2-3\varepsilon.
\end{eqnarray*}
Letting $\varepsilon\downarrow0$ we conclude that $\Psi$ satisfies
($\Phi $-DE).
\end{proof}

\begin{remark} \label{RN}
The requirement on the weak compactness of the function $\Psi$ can be
 substituted in the result above by the more general notion of
 Radon-Nikod\'ym function, which fits exactly with what is needed; see
 the definition and how to use it in this setting for example in 
 \cite{boseka}. 
One way of defining the Radon-Nikod\'ym property for a closed convex
 set  $A$ is that every closed convex subset $B \subset A$ is the
 closed convex hull of its 
 strongly exposed points. (See \cite[Th.~5.8 and Th.~5.17]{BL1}.)
So, a function is said to be a Radon-Nikod\'ym
 function if the closure of $T(B_X)$ has the Radon-Nikod\'ym
 property.  
\end{remark}

\begin{cor}
\label{cor:daugavet property weakly compact first corollary}
Let $X,Y$
be Banach spaces and assume $Y$ has the Daugavet property. Consider
$\Phi ,\Psi\in\ell_{\infty}(B_{X},Y)$ such that $\Phi $ is a quotient map
and $\|\Psi\|=1$. If $\calS_{\Psi}<\calS_{\Phi }$ and $\Psi$ is weakly compact,
then
\[
\|\Phi +\Psi\|=2.
\]
\end{cor}

\begin{proof}
In Lemma~\ref{lem:characterization of when T is a quotient map and Y has the daugavet property}
we observed that if $Y$ has the Daugavet property and $\Phi $ is a quotient
map, then $Y$ has the $\Phi $-Daugavet property with respect to
$Y^{*}\Phi \cdot Y$.
Thus, Theorem~\ref{thm:daugavet property weakly compact main theorem}
yields $\|\Phi +\Psi\|=2$.
\end{proof}

\begin{rem}
\label{rem:strong slice continuity cannot be removed ex 2}
The above
corollary is not valid if the condition $\calS_{\Psi}<\calS_{\Phi }$ is removed.
For instance, consider $\Phi ,\Psi:L_{1}[0,1]\oplus_{1}L_{2}[1,2]\to
L_{1}[0,1]$
given by $\Phi ((f,g))=f$ and $\Psi((f,g))(x)=g(x+1)$, where $(f,g)\in
L_{1}[0,1]\oplus_{1}L_{2}[1,2]$.
Then $\Psi$ is weakly compact and $\|\Phi \|=\|\Psi\|=1$, but $\|\Phi
+\Psi\|\leq1$.
\end{rem}

\begin{thm}
\label{thm:daugavet property weakly compact second main theorem}
Let
$X,Y$ be Banach spaces and let $\mathcal{Z}$ be a subspace of
$\ell_{\infty}(B_{X})$.
Assume $\Phi \in\ell_{\infty}(B_{X},Y)$ with $\|\Phi \|=1$. Then the following
are equivalent:
\begin{enumerate}
\item[(1)]
For every $x'\in\mathcal{Z}$ and every $y\in Y$,
$x'\otimes y$ satisfies $\mathrm{(}\Phi $-$\mathrm{DE)}$.
\item[(2)]
For every $x'\in S_{\mathcal{Z}}$, every $y\in S_{Y}$,
and every $\varepsilon>0$, there exist $\omega\in\mathbb{T}$ and
$x\in B_{X}$ such that
\[
\re \omega x'(x)\geq1-\varepsilon\quad\mbox{and}\quad\left\Vert \omega
  \Phi (x)+y\right\Vert \geq2-\varepsilon.
\]
\item[(3)]
Every weakly compact $\Psi\in\ell_\infty(B_X,Y)$ such that $y^*\circ
\Psi\in \mathcal{Z}$ for all $y^*\in Y^*$ satisfies
$\mathrm{(}\Phi $-$\mathrm{DE)}$.
\end{enumerate}
\end{thm}

\begin{proof}
(1) $\Leftrightarrow$ (2):
This equivalence follows from Theorem~\ref{thm:generalization daugavet
  center one map}.

(1) $\Rightarrow$ (3):
Let $\Psi$ be as in (3).  Because of (1), $Y$ has the $\Phi$-Daugavet
property with respect to $Y^{*}\Psi\cdot Y$. Since  trivially
$\calS_{\Psi}<\calS_{\Psi}$, 
Theorem~\ref{thm:daugavet property weakly compact main theorem} gives
(3).

(3) $\Rightarrow$ (1):
Given $x'\in\mathcal{Z}$ and $y\in Y$,
$x'\otimes y$ has finite-dimensional range and consequently
is a weakly compact map.
\end{proof}

For completeness we note the $n$-linear version of
\cite[Cor.~3.10]{SPWFuncAp}. 

\begin{cor}\label{cor3.16}
Let $X_{1},\dots,X_{n},Y$ be Banach spaces and consider a continuous
multilinear map $B_{0}:X_{1}\times\cdots\times X_{n}\to Y$ satisfying
$B_{0}(U_{X_{1}\times\cdots\times X_{n}})=U_{Y}$. Consider the subsets
$R$, $C$ and $WC$ of $L(Y,Y)$ of rank one, compact and weakly
compact linear operators. Denote $R\circ B_{0}=\{T \circ B_{0}:  T\in R\}$,
$C\circ B_{0}=\{T\circ B_{0}:  T\in C\}$ and $WC\circ
B_{0}=\{T\circ B_{0}:  T\in WC\}$.
Then the following are equivalent:
\begin{enumerate}
\item[(1)]
$Y$ has the Daugavet property.
\item[(2)]
$Y$ has the $B_{0}$-Daugavet property with respect to
$R\circ B_{0}$.
\item[(3)]
$Y$ has the $B_{0}$-Daugavet property with respect to
$C\circ B_{0}$.
\item[(4)]
$Y$ has the $B_{0}$-Daugavet property with respect to
$WC\circ B_{0}$.
\end{enumerate}
\end{cor}

\begin{proof}
The equivalence of (1) and (2) follows from
Lemma~\ref{lem:characterization of when T is a quotient map and Y has
  the daugavet property}.
(2) and (4) are equivalent by letting $\mathcal{Z}=\{y^{*}\circ
B_{0}:  y^{*}\in Y^{*}\}$
in Theorem~\ref{thm:daugavet property weakly compact second main theorem}.
The implications (4) $\Rightarrow$ (3) $\Rightarrow$ (2) are due
to the inclusions $R\subset C\subset WC$.
\end{proof}


\section{\label{chap:Weak-Slice-Continuity}Weak slice continuity}

In the last section we defined the notion of strong slice continuity
and related it to the Daugavet equation.
This section is the analogue of Section~\ref{chap:Strong-Slice-Continuity}
for the alternative Daugavet equation. We introduce the concept of
weak slice continuity to further investigate when two maps $\Psi,\Phi$
satisfy the alternative Daugavet equation, i.e., when
\[
\max_{|\omega|=1}\|\Phi+\omega\Psi\|=\|\Phi\|+\|\Psi\|.
\]

\begin{defn}
\label{def:weak slice}
Let $X$ be a Banach space, $x'\in\ell_{\infty}(B_{X})$
with $\|x'\|=1$ and $\varepsilon>0$. We write
\[
S'(x',\varepsilon)=\{x\in B_{X}: |x'(x)|\geq1-\varepsilon\}
\]
for the \textit{weak slice} of $B_{X}$ determined by $x'$ and $\varepsilon$.
\end{defn}

In a second step we extend the above definition to Banach space valued
functions.

\begin{defn}
\label{def:natural set of slices-1}
Let $X,Y$ be Banach spaces and
$\Phi\in\ell_{\infty}(B_{X},Y)$. The \textit{natural set of weak slices} defined
by $\Phi$ is given by
\[
\calS_{\Phi}'=\{S'(\Phi_{y^{*}},\varepsilon): 0<\varepsilon<1,\, y^{*}\in
Y^{*},\, y^{*}\Phi\neq0\}.
\]
\end{defn}

Now we are in a position to define weak slice continuity in
analogy to strong slice continuity; cf.\   Definition~\ref{def:slice
  continuity}.

\begin{defn}
Let $X,Y$ be Banach spaces and $\Phi,\Psi\in\ell_{\infty}(B_{X},Y)$. We
write $\calS_{\Psi}'<\calS_{\Phi}'$ if for every weak slice
$S'(\Psi_{z^{*}},\varepsilon)\in \calS_{\Psi}'$
there is another weak slice $S'(\Phi_{y^{*}},\mu)\in \calS_{\Phi}'$
such that
\[
S'(\Phi_{y^{*}},\mu)\subset S'(\Psi_{z^{*}},\varepsilon).
\]
In this case we say that $\Psi$ is \textit{weakly slice continuous} with respect
to $\Phi$.
\end{defn}

If $\Phi,\Psi$ are two maps such that $\Psi$ is strongly slice continuous
with respect to $\Phi$, then $\Psi$ is also slice continuous with respect
to $\Phi$. Let us check that a similar implication holds for strong
and weak slice continuity.

\begin{rem}
\label{rem:strong slice continuity implies weak slice continuity}
Let
$X,Y$ be Banach spaces and $\Phi,\Psi\in\ell_{\infty}(B_{X},Y)$. Then
$\calS_{\Psi}<\calS_{\Phi}$ implies $\calS_{\Psi}'<\calS_{\Phi}'$.
\end{rem}

\begin{proof}
Assume $S'(\Psi_{z^{*}},\varepsilon)\in \calS_{\Psi}'$. Since
$\calS_{\Psi}<\calS_{\Phi}$, there is $S(\Phi_{y^{*}},\mu)\in
\calS_{\Phi}$ satisfying 
$S(\lambda \Phi_{y^{*}},\mu)\subset S(\lambda \Psi_{z^{*}},\varepsilon)$
for all $\lambda\in\mathbb{T}$. We claim $S'(\Phi_{y^{*}},\mu)\subset
S'(\Psi_{z^{*}},\varepsilon)$.
To prove this, let $x\in B_{X}$ with $|\Phi_{y^{*}}(x)|\geq1-\mu$ and
denote $\omega=|\Phi_{y^{*}}(x)|/\Phi_{y^{*}}(x)$. Then $\re \omega
\Phi_{y^{*}}(x)=|\Phi_{y^{*}}(x)|\geq1-\mu$
and therefore $\re \omega \Psi_{z^{*}}(x)\geq1-\varepsilon$.
In particular, $|\Psi_{z^{*}}(x)|\geq1-\varepsilon$, i.e., $x\in
S'(\Psi_{z^{*}},\varepsilon)$.
\end{proof}

The next example shows that the reverse implication in the above remark
does not hold.

\begin{example}
\label{exa:finally some good example}
Let $\Psi:\mathbb{R}\to\mathbb{R}$ be defined by
\[
\Psi(x)=\begin{cases}
1 & \mbox{if }x=0,\\
-|x| & \mbox{if }x\neq0.
\end{cases}
\]
Then $\Psi$ is weakly slice continuous with respect to the identity,
but $\Psi$ is not strongly slice continuous with respect to the
identity.
\end{example}

\begin{proof}
Consider the slice $S(\Psi,1/2)\in \calS_{\Psi}$. Then
$S(c\,\mathrm{Id},\varepsilon)\not\subset S(\Psi,1/2)$
for any $c\in\{-1,1\}$ and $0<\varepsilon<1$. Thus $\Psi$ is not
strongly slice continuous with respect to $\mathrm{Id}$. But if $c\in\{-1,1\}$
and $0<\varepsilon<1$ are given, then
$S'(\mathrm{Id},\varepsilon)\subset S'(c\Psi,\varepsilon)$.
Therefore $\Psi$ is weakly slice continuous with respect to
$\mathrm{Id}$.
\end{proof}

\begin{example}
Let $X,Y$ be Banach spaces. Consider $\Phi\in\ell_{\infty}(B_{X},Y)$
and a bounded linear operator $P:Y\to Y$. Denote $\Psi=P\circ \Phi$. Then
$\calS_{\Psi}'<\calS_{\Phi}'$.
\end{example}

\begin{proof}
According to Example~\ref{exa:strong slice continuity concatenation},
the assumptions imply $\calS_{\Psi}<\calS_{\Phi}$. Hence
$\calS_{\Psi}'<\calS_{\Phi}'$ 
by Remark~\ref{rem:strong slice continuity implies weak slice continuity}.
\end{proof}

Note that we have shown in Example~\ref{exa:finally some good example}
that there are bounded maps $\Phi,\Psi$ with $\calS_{\Psi}'<\calS_{\Phi}'$,
but $\Psi\neq P\circ \Phi$ for any bounded linear operator $P$.

Recall from Definition~\ref{def:quotient map} that a quotient map
is a continuous function mapping the open unit ball of its domain
onto the open unit ball of its range space. These properties allow
for the following lemma.
\begin{lem}
\label{lem:characterization of when T is a quotient map and Y has the
  alternative daugavet property}
Let
$X,Y$ be Banach spaces and assume $\Phi\in\ell_{\infty}(B_{X},Y)$ is
a quotient map. Then the following are equivalent:
\begin{enumerate}
\item[(1)]
$Y$ has the alternative Daugavet property.
\item[(2)]
$Y$ has the alternative $\Phi$-Daugavet property with respect
to $Y^{*}\Phi\cdot Y$.
\end{enumerate}
\end{lem}

\begin{prop}
\label{prop:max omega norm of T plus omega y star R ot y equals 2}
Let
$X,Y$ be Banach spaces and assume $Y$ has the alternative Daugavet
property. Consider $\Psi,\Phi \in\ell_{\infty}(B_{X},Y)$ such that $\Phi $
is a quotient map and $\|\Psi\|=1$. Then $\calS_{\Psi}'<\calS_{\Phi }'$
implies that for every $y\in S_{Y}$ and $y^{*}\in Y^{*}$ with $y^{*}\Psi\neq0$
\[
\max_{|\omega|=1}\|\Phi +\omega \Psi_{y^{*}}\otimes y\|=2.
\]
\end{prop}

\begin{proof}
We will use Theorem~\ref{thm:generalization alternative daugavet center one map}, 
i.e., we need to show that for every $\varepsilon>0$ there exist
$\omega\in\mathbb{T}$ and $x\in S'(\Psi_{y^{*}},\varepsilon)$
such that
\[
\left\Vert \omega \Phi (x)+y\right\Vert \geq2-\varepsilon.
\]

Since $\calS_{\Psi}'<\calS_{\Phi }'$, there is a slice $S'(\Phi
_{z^{*}},\mu)\in \calS_{\Phi }'$ 
such that $S'(\Phi _{z^{*}},\mu)\subset S'(\Psi_{y^{*}},\varepsilon)$
and $\mu\leq\varepsilon$. The alternative Daugavet property of $Y$
in conjunction with Lemma~\ref{lem:characterization of when T is a quotient map and Y has the alternative daugavet property}
yields the norm equality $\max_{|\omega|=1}\|\Phi +\omega \Phi
_{y^{*}}\otimes y\|=2$. 
Hence another application of 
Theorem~\ref{thm:generalization alternative daugavet center one map}
gives $\omega\in\mathbb{T}$ and $x\in S'(\Phi _{y^{*}},\mu)$
such that
\[
\left\Vert \omega \Phi (x)+y\right\Vert \geq2-\mu\geq2-\varepsilon.
\]

Because of $S'(\Phi _{z^{*}},\mu)\subset S'(\Psi_{y^{*}},\varepsilon)$,
we also have $x\in S'(\Psi_{y^{*}},\varepsilon)$, which completes
the proof.
\end{proof}

\begin{rem}
In the above proposition, the assumption $\calS_{\Psi}'<\calS_{\Phi }'$
cannot be removed. This can be shown by using the functions from
Remark~\ref{rem:strong slice continuity cant be removed ex 1}.
\end{rem}

\begin{thm}
\label{thm:alternative daugavet property weakly compact main theorem}
Let
$X,Y$ be Banach spaces and let $\Phi,\Upsilon,\Psi\in\ell_{\infty}(B_{X},Y)$
with $\|\Phi \|=\|\Upsilon\|=\|\Psi\|=1$. Assume $Y$ has the alternative
$\Phi $-Daugavet
property with respect to $Y^{*}\Upsilon\cdot Y$. Then, if
$\calS_{\Psi}'<\calS_{\Upsilon}'$ 
and $\Psi$ is weakly compact,
\[
\max_{|\omega|=1}\|\Phi +\omega\Psi\|=2.
\]
\end{thm}

\begin{proof}
Denote $K=\overline{\mathrm{co}}(\mathbb{T}\Psi(B_{X}))$ and let
$\varepsilon>0$ be given. In the same way as in the proof of Theorem
\ref{thm:daugavet property weakly compact main theorem}, we may find
$y_{0}\in K$ with $\|y_{0}\|>1-\varepsilon$, $\delta\in(0,\varepsilon/2)$
and $y_{0}^{*}\in Y^{*}$ such that
\[
y\in K,\,\re
y_{0}^{*}(y)\geq1-\delta\quad\Rightarrow\quad\|y-y_{0}\|<\varepsilon
\]
and
\[
\sup_{y\in K}|y_{0}^{*}(y)|=1.
\]
Setting $\psi:=y_{0}^{*}\circ\Psi$, we get
\[
\|\psi\|=\sup_{x\in B_{X}}|y_{0}^{*}(\Psi(x))|=\sup_{y\in K}|y_{0}^{*}(y)|=1,
\]
i.e., $S'(\psi,\delta)\in \calS_{\Psi}'$. From
$\calS_{\Psi}'<\calS_{\Upsilon}'$ 
we deduce the existence of $\mu\leq\delta$ as well as
$S'(\Upsilon_{z^{*}},\mu)\in \calS_{\Upsilon}'$
satisfying $S'(\Upsilon_{z^{*}},\mu)\subset S'(\psi,\delta)$.
Since $Y$ has the alternative $\Phi $-Daugavet property with respect
to $Y^{*} \Upsilon \cdot Y$, we can use Theorem~\ref{thm:generalization
  alternative daugavet center one map}
to get $\omega_{1}\in\mathbb{T}$ and $x\in S'(\Upsilon_{z^{*}},\mu)$
such that
\[
\left\Vert \omega_{1}\Phi (x)+\frac{y_{0}}{\|y_{0}\|}\right\Vert
\geq2-\mu\geq2-\varepsilon.
\]
In particular, $x\in S'(\Upsilon_{z^{*}},\mu)\subset S'(\psi,\delta)$.
Writing $\omega_{2}=|\psi(x)|/\psi(x)$ we observe
\[
\re y_{0}^{*}(\omega_{2}\Psi(x))=\re \omega_{2}\psi(x)=|\psi(x)|\geq1-\delta,
\]
so the fact that $\omega_{2}\Psi(x)\in K$ gives
\[
\|\omega_{2}\Psi(x)-y_{0}\|<\varepsilon.
\]
On the other hand,
\begin{eqnarray*}
\|y_{0}+\omega_{1}\Phi (x)\| & \geq & \left\Vert \omega_{1}\Phi
 (x)+\frac{y_{0}}{\|y_{0}\|}\right\Vert -\left\Vert
 y_{0}-\frac{y_{0}}{\|y_{0}\|}\right\Vert \\
 & = & \left\Vert \omega_{1}\Phi
 (x)+\frac{y_{0}}{\|y_{0}\|}\right\Vert -|\,\|y_{0}\|-1|\\
 & \geq & 2-2\varepsilon.
\end{eqnarray*}
Altogether
\begin{eqnarray*}
\max_{|\omega|=1}\|\Phi +\omega\Psi\| & \geq & \|\Phi
 +\overline{\omega_{1}}\omega_{2}\Psi\|\\
 & \geq & \|\Phi (x)+\overline{\omega_{1}}\omega_{2}\Psi(x)\|\\
 & = & \|\omega_{1}\Phi (x)+\omega_{2}\Psi(x)\|\\
 & \geq & \|\omega_{1}\Phi (x)+y_{0}\|-\|\omega_{2}\Psi(x)-y_{0}\|\\
 & \geq & 2-3\varepsilon,
\end{eqnarray*}
which proves the assertion because $\varepsilon>0$ was chosen
arbitrarily.
\end{proof}

\begin{cor}
\label{cor:daugavet property weakly compact first corollary-1}Let
$X,Y$ be Banach spaces and assume $Y$ has the alternative Daugavet
property. Consider $\Phi ,\Psi\in\ell_{\infty}(B_{X},Y)$ such that $\Phi $
is a quotient map and $\|\Psi\|=1$. If $\calS_{\Psi}'<\calS_{\Phi }'$
and $\Psi$ is weakly compact, then
\[
\max_{|\omega|=1}\|\Phi +\omega\Psi\|=2.
\]
\end{cor}

\begin{proof}
$Y$ has the alternative $\Phi $-Daugavet property with respect to
$Y^{*}\Phi \cdot Y$
by Lemma~\ref{lem:characterization of when T is a quotient map and Y has the alternative daugavet property}.
Therefore $\Psi$ satisfies the alternative $\Phi $-Daugavet equation
according to Theorem~\ref{thm:alternative daugavet property weakly
  compact main theorem}.
\end{proof}

\begin{rem}
In the above corollary, the assumption $\calS_{\Psi}'<\calS_{\Phi }'$
cannot be dropped. For instance, this follows with the help of the
functions constructed 
in Remark~\ref{rem:strong slice continuity cannot be removed ex 2}.
\end{rem}

\begin{thm}
\label{thm:alternative daugavet property weakly compact second main
  theorem}
Let
$X,Y$ be Banach spaces and let $\mathcal{Z}$ be a subspace of
$\ell_{\infty}(B_{X})$.
Assume $\Phi \in\ell_{\infty}(B_{X},Y)$ with $\|\Phi \|=1$. Then the following
are equivalent:
\begin{enumerate}
\item[(1)]
For every $x'\in\mathcal{Z}$ and every $y\in Y$,
$x'\otimes y$ satisfies $\mathrm{(}\Phi $-$\mathrm{ADE)}$.
\item[(2)]
For every $x'\in S_{\mathcal{Z}}$, every $y\in S_{Y}$,
and every $\varepsilon>0$, there exist $\omega_{1},\omega_{2}\in\mathbb{T}$
and $y\in B_{X}$ such that
\[
\re \omega_{1}x'(x)\geq1-\varepsilon\quad\mbox{and}\quad\left\Vert
  \omega_{2}\Phi (x)+y\right\Vert \geq2-\varepsilon.
\]
\item[(3)]
For every $x'\in S_{\mathcal{Z}}$, every $y\in S_{Y}$,
and every $\varepsilon>0$, there exist $\omega\in\mathbb{T}$ and
$x\in B_{X}$ such that
\[
\left|x'(x)\right|\geq1-\varepsilon\quad\mbox{and}\quad\left\Vert
  \omega \Phi (x)+y\right\Vert \geq2-\varepsilon.
\]
\item[(4)] 
Every weakly compact $\Psi\in\ell_\infty(B_X,Y)$ such that $y^*\circ
\Psi\in \mathcal{Z}$ for all $y^*\in Y^*$ satisfies
$\mathrm{(}\Phi $-$\mathrm{ADE)}$.
\end{enumerate}
\end{thm}

\begin{proof}
The equivalence of (1), (2) and (3) is a consequence of
Theorem~\ref{thm:generalization alternative daugavet center one map}.
The implication (1) $\Rightarrow$ (4) follows from
Theorem~\ref{thm:alternative daugavet property weakly compact main theorem}
since trivially $\calS_{\Psi}'<\calS_{\Psi}'$ for any
$\Psi$ as in~(4).
The direction (4) $\Rightarrow$ (1) is true because finite-rank maps
are weakly compact.
\end{proof}

The following corollary is analogous to Corollary~\ref{cor3.16}. 

\begin{cor}
Let $X_{1},\dots,X_{n},Y$ be Banach spaces and consider a continuous
multilinear map $B_{0}:X_{1}\times\cdots\times X_{n}\to Y$ satisfying
$B_{0}(U_{X_{1}\times\cdots\times X_{n}})=U_{Y}$. Consider the subsets
$R$, $C$ and $WC$ of $L(Y,Y)$ of rank one, compact and weakly
compact operators. Denote $R\circ B_{0}=\{T\circ B_{0}:  T\in R\}$,
$C\circ B_{0}=\{T\circ B_{0}:  T\in C\}$ and $WC\circ
B_{0}=\{T\circ B_{0}:  T\in WC\}$.
Then the following are equivalent:
\begin{enumerate}
\item[(1)]
$Y$ has the alternative Daugavet property.
\item[(2)]
$Y$ has the alternative $B_{0}$-Daugavet property with
respect to $R\circ B_{0}$.
\item[(3)]
$Y$ has the alternative $B_{0}$-Daugavet property with
respect to $C\circ B_{0}$.
\item[(4)]
$Y$ has the alternative $B_{0}$-Daugavet property with
respect to $WC\circ B_{0}$.
\end{enumerate}
\end{cor}

\begin{proof}
The equivalence of (1) and (2) is due to Lemma~\ref{lem:characterization of when T is a quotient map and Y has the alternative daugavet property}.
(2) and (4) are equivalent by letting $\mathcal{Z}=\{y^{*}\circ
B_{0}:  y^{*}\in Y^{*}\}$
in Theorem~\ref{thm:alternative daugavet property weakly compact
  second main theorem}.
The implications (4) $\Rightarrow$ (3) $\Rightarrow$ (2) follow
from the inclusions $R\subset C\subset WC$.
\end{proof}


\section{Local $\Phi$-Daugavet type properties and applications}

 After the explanation of our main results given in the previous part
 of the paper, we are ready to present more technical versions of the
 tools obtained there. All of them can be  proved using the same
 arguments and are useful for  applications. 
 Essentially, we introduce  the notion of norm determining set $\Gamma
 \subset B_X$ for a class of functions and some new elements that
 allow to define the notion of $\Phi$-Daugavet property with respect
 to particular sets of scalar functions and vectors in $Y$, with a
 norm that can be defined as the supremum of the evaluation of the
 functions involved just for some subset of vectors in $S_X$. 

In this section all Banach spaces are supposed to be $\R$-vector
  spaces for simplicity of notation. 

  Let $X$ and $Y$ be Banach spaces, and let $V \subset
\ell_\infty(B_X,Y)$.
We say that a subset $\Gamma \subset B_{X}$ is \textit{norm determining}
for $V$ if
$$
\| \Psi\| = \|\Psi\|_\Gamma := \sup_{x \in \Gamma} \| \Psi(x)\|
$$
for all $\Psi \in V$.

 Let us start by formulating a version of Theorem \ref{thm:generalization
 daugavet center one map} considering norm determining subsets for the
 functions involved. Its proof follows the same lines as the proof
 of that theorem, so we omit it.

\begin{proposition} \label{geom0}
Let $X$ and $Y$ be  Banach spaces and let $\Gamma \subset B_X$. Let
$\Phi \in \ell_\infty(B_X,Y)$ be a norm one 
map, and consider a norm one function $x' \in \ell_\infty(B_X)$. Let
$y \in S_Y$. The following assertions are
equivalent.
\begin{itemize}
\item[(1)]
$ \| \Phi+ x' \otimes y\|_\Gamma=2$.
\item[(2)]
For every $\varepsilon>0$  there is  some
$\omega \in \mathbb T$ and an element  $x \in S(\omega x',\varepsilon)
\cap \Gamma$  such that
$$
\| \omega \Phi(x) + y\| \ge 2- 2 \varepsilon.
$$
\end{itemize}
\end{proposition}

\begin{remark}
Notice that the condition in the  result above  implies that for  a
norm one scalar
function $x' \in \ell_\infty(B_X,Y)$,
 $$
 2 \le \|x' \otimes y + \Phi\|_\Gamma \le \|x'\|_\Gamma \|y\| + \|\Phi\| \le
 \|x'\|_\Gamma +1,
 $$
 and so $\|x'\|_\Gamma = \|x'\| =1$. Thus $\Gamma$ is norm determining
 for $x'$; the same argument gives that it is so for $\Phi$. 
\end{remark}

Let us define  now some sort of ``local version" of the notion of
$\Phi$-Daugavet property. 

\begin{definition} \label{definu}
Let $X$ and $Y$ be  Banach spaces and let $\Phi:B_X \to Y$ be a norm one
function. Let $\Gamma \subset B_X$ be a norm determining  set for $\Phi$ and
consider subsets $\mathcal W \subset \ell_\infty(B_X)$ and $\Delta
\subset S_Y$. We say that  $Y$ has 
the $\Phi$-Daugavet property with
respect to $(\Gamma,\mathcal W, \Delta)$ if for every $x' \in \mathcal
W$ and $y \in \Delta$
$$
\sup_{x \in \Gamma} \|\Phi(x)+x'(x)y\|=2.
$$
\end{definition}

The reader can notice that this definition is related to the one of
Daugavet centre given in
Definition~1.2 of \cite{boseka} and that of the almost \DP\ from
\cite{houston2011}. 

Let us provide a concrete example of a function $\Phi$ and sets $\Gamma$,
$\mathcal W$ and $\Delta$ for which every Banach space has the
$\Phi$-Daugavet property with respect to $(\Gamma,\mathcal W,
\Delta)$. 

\begin{example} \label{exsim}
Let $X$ be a real Banach space and take $Y=X$. Consider the sets $\Gamma = B_X$,
$$
\mathcal W= \{x' \in \ell_{\infty}(B_X): |x'(x)|=1 \
\text{and}\ x'(x)=x'(-x) 
\ \text{for all} \ x \in S_X \},
$$
and $\Delta= S_X$. Let $\Phi:B_X \to X$ be a  norm one function such
that $\Phi(S_X)=S_X$ and $\Phi(-x)=-\Phi(x)$ for all~$x$. Take $\varepsilon
>0$. Fix a norm one function 
$x' \in \mathcal W$. If $y \in S_X$, take $x_0 \in S_X$ such
that $\Phi(x_0)= y$. If $x'(x_0)=1$,  then
$$
\sup_{x \in \Gamma} \| \Phi(x) + x'(x) y \|  \ge \| \Phi(x_0) + x'(x_0)
y \| \ge 2 \|y\|=2.
$$
In case $x'(x_0)=-1=x'(-x_0)$, then $\Phi(-x_0)=- \Phi(x_0)= -y$, and thus
$$
\sup_{x \in \Gamma} \| \Phi(x) + x'(x) y \| \ge \| \Phi(-x_0) + x'(-x_0)
y \| \ge \|{-}y-y\| =2.
$$
Therefore, $X$ has the $\Phi$-Daugavet property with respect to
$(\Gamma,\mathcal W, \Delta)$.

The space $\ell^\infty$ and the function $\Phi(x)=x^3$ show an example
of this situation, although $\ell^\infty$ does not have the Daugavet
property.
\end{example}

The proof of the following result is a direct application of
Proposition~\ref{geom0}.

\begin{corollary} \label{geom1}
Let $X$ and $Y$ be  Banach spaces and consider $\Phi$, $\Gamma$, $\mathcal W$
and $\Delta$ as in Definition \ref{definu}.  The following statements
are equivalent.
\begin{itemize}
\item[(1)]
$Y$ has the $\Phi$-Daugavet property with respect to $(\Gamma,\mathcal
W, \Delta)$.
\item[(2)]
For every $y \in \Delta$, for every $x' \in \mathcal W$ of norm one 
and for every
$\varepsilon >0$ there are $\omega \in \mathbb T$ and an element $x \in
S(\omega x',\varepsilon) \cap \Gamma$
such that
$$
\|\omega \Phi(x)+y\| \ge 2- 2 \varepsilon.
$$
\end{itemize}
\end{corollary}

\begin{remark} \label{compo}
Let us show that, under the assumption that the function $\Phi$ maps
$B_X$ onto $B_Y$,
the Daugavet property for $Y$ implies the
$\Phi$-Daugavet property with respect to $\Gamma = B_X$, $\mathcal W=
\{x':X \to \mathbb R : x'= y^* \circ \Phi, \, y^* \in S_{Y^*} \}$ and
$\Delta= S_Y$. This case is  canonical, and in a sense also trivial,
since the result is a consequence of some simple
computations. However, there are more examples that show that not all
the cases can be obtained in this way, i.e., there are families of
functions $\mathcal W$ whose elements are not compositions of a given
$\Phi$ and the functionals of $S_{Y^*}$ for which $\Phi$ satisfies
the Daugavet equation.

(1) Let us first show the statement announced above.
Let $Y$ be a Banach space with the Daugavet property and let $\Phi:B_X
\to Y$ satisfy $\Phi(B_X)=B_Y$.
 Let us show that then $Y$ has the
$\Phi$-Daugavet property with respect to $(B_X, \, \mathcal W,
\,S_Y)$, where $\mathcal W = \{x':X \to \mathbb R : x'= y^* \circ \Phi$, 
$y^* \in S_{Y^*} \}$.

To see this, suppose that $\Phi:B_X \to Y$ satisfies
$\Phi(B_X)=B_Y$. Then we claim that for each $\varepsilon >0$, $y^*
\in S_{Y^*}$ and $y \in S_Y$ there is $x \in S(y^* \circ \Phi,
\varepsilon)$ such that
$$
\| \Phi(x) + y \| \ge 2 - 2 \varepsilon.
$$
Indeed, let $\varepsilon >0$, $y \in S_Y$ and $y^* \in
S_{Y^*}$. Then by the Daugavet property for $Y$ there is an element $z
\in S(y^*,\varepsilon)$ such that $\|z+y\| \ge 2 - 2
\varepsilon$. Since $\Phi$ maps $B_X$ onto $B_Y$, we
find $x \in B_X$ such that $\Phi(x)=z \in S(y^*, \varepsilon)$, and so
$\langle \Phi(x), y^* \rangle = y^* \circ \Phi(x) > 1 - \varepsilon$
and $\| \Phi(x) + y \| \ge 2 - 2 \varepsilon$.
An application of Corollary~\ref{geom1} gives the result.

(2) There are also other families of functions $\mathcal W$ for which
    the Daugavet equation is satisfied with a function $\Phi$, but
    they cannot be defined by composition as in (1). For example, take
    $X=Y=C(K)$, 
    where $K$ is a perfect compact Hausdorff space,
    and define $\mathcal W$ as the set of continuous
linear functionals on this space. Consider the function
    $x\mapsto \Phi(x)=x^3$. Clearly, a linear 
    functional cannot be written as a composition of $\Phi$ and some other
    linear functional. However, for each norm one element $y \in
    S_{C(K)}$ we find an element $x \in S_{C(K)}$ such that
    $x^3=y$. This, together with the Daugavet property of $C(K)$,
    implies (2) in Corollary~\ref{geom1}. 
     To see this, just
    take into account that by the Daugavet property of $C(K)$, for each
    $\varepsilon >0$, each $y \in S_{C(K)}$ and each $y^* \in
    S_{C(K)^*}$ there is $x \in S(y^*,\varepsilon/2)$ such that
    $$
    \|x+y\|
    > 2 - 2 ( \varepsilon / 2)= 2- \varepsilon > 2- 2 \varepsilon.
    $$
    Take $z \in S_{C(K)}$ such that
    $z^3=x$, and so
    $\| z^3 + y \| > 2- 2 \varepsilon$.
    Let us show that  $z \in S(y^*,\varepsilon)$ too, that is, (2) in
    Corollary~\ref{geom1} 
    holds.
    Consider the measurable sets defined as $A^+:=\{w \in K: z(w) \ge 0
    \}$ and $A^-:=\{w \in K: z(w) < 0 \}$. Also take the decomposition
    of the measure $\mu$ that defines the functional $y^*$ as a
    difference of positive disjointly supported measures  $\mu= \mu^+ -
    \mu^-$. 
    Then, using that $|z^3| \le |z|$, we get
\bea
    1 - \varepsilon/2 &\le & 
\int_K z^3 \, d \mu \\
&=& 
\int_{A^+} |z^3|\, d \mu^+ + \int_{A^-} |z^3|\, d \mu^-
     - \int_{A^+} |z^3|\, d \mu^-    - \int_{A^-} |z^3| \,d \mu^+\\
&\le&  
\int_{A^+} |z| \,d \mu^+ + \int_{A^-} |z| \,d \mu^- \\
&\le&
\mu^+(A^+) + \mu^-(A^-) \le 1.
\eea
Hence
     $$
      \mu^+(A^-) + \mu^-(A^+) \le \varepsilon/2.
     $$
      Consequently,
\bea
1 &\ge& \int_K z \,d \mu \\
&=& 
\int_{A^+} |z|\, d \mu^+ + \int_{A^-} |z|\, d \mu^-
     - \int_{A^+} |z| \,d \mu^-    - \int_{A^-} |z| \,d \mu^+ \\
&\ge& 
(1- \varepsilon/2) - (\mu^-(A^+) + \mu^+(A^-)) \\
&\ge & 
1 - 2 (\varepsilon/2) = 1- \varepsilon.
\eea
    Then $z \in S(y^*,\varepsilon)$, and we get the result.

 (3) Surjectivity of $\Phi$ is sometimes not needed if the
     sets   $\Gamma$, $\mathcal W$ and $\Delta$ are adequately chosen.
     Take now $X=Y=C(K)$, $\Phi (x)=|x|^{1/4}$ and $\mathcal W$  the
     set of probability measures $\mathcal P(K) \subset
     C(K)^*$. Take also $\Gamma = B_{C(K)^+}$ and $\Delta =
     S_{C(K)^+}$.  Then the $\Phi$-Daugavet property with respect to
     $(\Gamma,\mathcal W, \Delta)$ is satisfied, as a consequence of
     Corollary~\ref{geom1}.  To see this, note that if $y \in
     S_{C(K)^+}$ and $\mu \in \mathcal P(K)$, then for $\omega =1$ we
     obtain by the Daugavet property of $C(K)$, given $\eps>0$,  
a (positive) function
     $x$ of norm one in $S_{C(K)}$ such that
     $\int_K x \,d \mu \ge  1 - \varepsilon$ and $\| x + y \| \ge 2 - 2
     \varepsilon$. Then since $1 \ge x^{1/4} \ge x$ we obtain
     $$
     \| x^{1/4} + y \| \ge \| x + y\| \ge 2- 2 \varepsilon,
     $$
     i.e., the $\Phi$-Daugavet property with respect to
     $(\Gamma,\mathcal W, \Delta)$ is satisfied. Again, the elements
     of $\mathcal P(K) $ cannot be factored through $\Phi$.
\end{remark}

The following result gives the main tool for extending the Daugavet
equation to other functions not belonging to the set of products  of
scalar functions of $\mathcal W$ and elements of the unit sphere of
$Y$. In particular, well-known arguments provide the condition of
the following  theorem on the inclusion of the image of a slice in a
small ball for
the big class of the strong Radon-Nikod\'ym operators, which contains the
weakly compact ones (see  for example the first part of
\cite{ams2000}, or Theorem~1.1 in \cite{pol} for a version directly
related with the present paper).

\begin{theorem} \label{T1}
Let $\Psi: B_X \to Y$ be a norm one function.
If the Banach space $Y$ has the $\Phi$-Daugavet property with respect
to $(\Gamma, \mathcal W, \Delta)$ for $\mathcal W \subset
\ell_\infty(B_X)$, and for 
all $\varepsilon >0$ there are $x' \in \mathcal W$, $\delta >0$ and $y
\in \Delta$ such that for all $\omega \in \mathbb T$,
$\Psi(S(\omega x',\delta) \cap \Gamma) \subset B_\varepsilon
(\overline{\omega} y)$, 
then
$$
\|\Phi + \Psi\|_\Gamma =2.
$$
\end{theorem}

\begin{proof}
Fix $\varepsilon >0$. By hypothesis there are $x'_0 \in \mathcal W$ and
$ y \in \Delta$ such that for every $\omega \in \mathbb T$, $\|\Psi(x) -
\overline{\omega} y\| < \varepsilon$ for all $x \in S(\omega x'_0,
\delta) \cap \Gamma$.

  By Corollary~\ref{geom1}, for  $x'_0$ and $y$ there is $\omega_0 \in
  \mathbb T$ and $x_0 \in S(\omega_0 x'_0,\delta) \cap \Gamma$ such
  that $\|\omega_0 \Phi(x_0)+ y\| \ge 2 - 2 \varepsilon$.
   Then,
\bea
 \| \Phi + \Psi\|_{\Gamma} &\ge& \| \Phi(x_0) + \Psi(x_0) \|\\
&\ge& \| \Phi(x_0) + \overline{\omega_0} y \| - 
\|\Psi(x_0) -   \overline{\omega_0} y \| \\
 &=& 
\|\omega_0 \Phi(x_0) + y\| - \|\Psi(x_0) - \overline{\omega_0}   y\| \\
&\ge& 2- 2 \varepsilon - \varepsilon
= 2- 3 \varepsilon.
\eea
 Since this happens for each $\varepsilon >0$, we obtain that $\|
  \Phi+ \Psi\|_\Gamma = 2$.
 \end{proof}

The proof of the following corollary is just an application of
Theorem~\ref{T1} for  $\mathcal W   :=\{x'=y^* \circ \Psi :  y^*\in
Y^*$,  
$\|y^* \circ \Psi\|=1\}$, $\Gamma = B_X$ and $\Delta= S_X$, together with
the argument in the proof of Theorem~\ref{thm:daugavet property weakly
  compact main theorem} regarding weakly compact sets that gives the
condition for applying Theorem~\ref{T1}. The same comments regarding
Radon-Nikod\'ym functions given in Remark~\ref{RN} apply in the
present case. 

\begin{corollary}
Let $\Phi:B_X \to Y$ be a norm one function such that $\Phi(B_X)=B_Y$ and
let $\Psi:B_X \to Y$ be a norm one weakly compact function. 
Suppose $Y$ has the $\Phi$-Daugavet property.
Then $\|\Phi + \Psi\|=2$.
\end{corollary}


We finish the paper by showing some special new tools for obtaining
applications in the case of $C(K)$-spaces and $L^1(\mu)$-spaces. 

\subsection{A general test for the $\Phi$-Daugavet property: the case of
 functions on $C(K)$-spaces}

The requirement $\Psi(S(\omega x',\delta) \cap \Gamma) \subset
B_\varepsilon (\overline{\omega} y)$ 
in  Theorem~\ref{T1} seems to be a difficult property to
check directly . The next result provides a simpler test that can be used in
some cases.  We will use
this new tool to analyse the Daugavet equation for functions on
$C(K)$-spaces.

\begin{proposition} \label{test}
Let $X$ be a Banach space.
Let $z \in S_{X}$, $K >0$  and let $\Phi, \Psi:B_X \to X$ be norm one
functions. Take a subset $B \subset B_X$. The following statements are
equivalent.
\begin{itemize}
\item[(1)]
There is a w$^*$-compact convex set $V \subset {X^*}$
such that for all finite sequences $x_1,\dots ,x_n \in B$
 and positive scalars $\alpha_1,\dots ,\alpha_n$ such that $\sum_{i=1}^n
 \alpha_i=1$ we have
$$
\sum_{i=1}^n \alpha_i \|\Psi(x_i) - z\| \le K \sup_{x^* \in V} \Bigl( 1-
\Bigl\langle \sum_{i=1}^n \alpha_i\Phi(x_i),x^* \Bigr\rangle \Bigr).
$$
\item[(2)]
For each $\varepsilon >0$ there exists $x^*_0 \in V$
such that
$$
\|\Psi(x) - z\| \le K (1- \langle \Phi(x),x^*_0 \rangle)
$$
for all $x \in B_X$.
\end{itemize}
These equivalent properties imply that for each $\varepsilon >0$ there
exists $x^*_0 \in V$ such that $\Psi(S(x^*_0 \circ \Phi,\varepsilon) \cap
B)  \subset B_{K \varepsilon}(z)$.
\end{proposition}

\begin{proof}
We shall obtain this result as a consequence of Ky Fan's lemma (see
\cite[p.~40]{Pie}), so it is in essence a consequence of the
Hahn-Banach theorem.  

We only sketch the proof. Consider the
concave set of convex functions $\Upsilon:V \to \mathbb R$ defined by
$$
\Upsilon(x^*):=\sum_{i=1}^n \alpha_i \|\Psi(x_i) - z\| - K \Bigl( 1-
\Bigl\langle 
\sum_{i=1}^n \alpha_i\Phi(x_i),x^* \Bigr\rangle\Bigr),
$$
where $\alpha_i > 0$, $\sum_{i=1}^n \alpha_i =1$ and $x_1,\dots ,x_n \in B$.
The inequality in (1) provides for such a $\Upsilon$ an element
$x^*_\Upsilon \in V$ such that $\Upsilon(x^*_\Upsilon) \le 0$. Ky Fan's Lemma
gives an element $x^*_0 \in V$ such that $\Upsilon(x^*_0) \le 0$ for all
the functions $\Upsilon$ in the family. This proves (1) $\Rightarrow$ (2),
and the converse is obvious.

On the other hand, if
$x \in S(x^*_0 \circ \Phi,\varepsilon) \cap B$, then
$$
\|R(x) - z \| \le K (1- \langle \Phi(x), x^*_0 \rangle ) \le K \varepsilon.
$$
This proves the final statement.
\end{proof}

\begin{example}
Let us show an application of the criterion given in
Proposition~\ref{test}. Let $X=C(K)$ and $V= B_{C(K)^*}$.
Take a positive norm one function $f$ in $C(K)$.
Define the class of functions $C$ by
$$
C=\{  g \in B_{C(K)}: g^2 \le f \le |g| \}.
$$
Let us see that the requirements of Proposition~\ref{test} are
satisfied for $B=C$ and $\Phi$ and $\Psi$ defined by $ \Phi(g)=g^2$
and $\Psi(g)=|g|$. Note that for all positive
functions $h \in B_{C(K)}$, $\mathbf{1}- h \le \mathbf{1}- h^2$. Then
for all $g_1,\dots, 
g_n \in C$ and positive $\alpha_1,\dots ,\alpha_n$ such that
$\sum_{i=1}^n \alpha_i=1$, we obtain
\bea
\sum_{i=1}^n \alpha_i \| |g_i| -\mathbf{1} \| &\le&
\sum_{i=1}^n \alpha_i \| \mathbf{1}- f \| = \| \mathbf{1}- f\|
\le
 \sum_{i=1}^n \alpha_i \|\mathbf{1}- g_i^2 \| \\
&\le& \sup_{x^* \in B_{C(K)^*}} \Bigl(1 - \Bigl\langle \sum_{i=1}^n \alpha_i
g_i^2, x^* \Bigr\rangle\Bigr).
\eea
Consequently, an application of the proposition shows that for each
$\varepsilon >0$ there exists $x^*_0 \in C(K)^* $ such that $\Psi(S(x^*_0
\circ \Phi,\varepsilon) \cap C)  \subset B_{K \varepsilon}(\mathbf{1})$.
\end{example}

Note that for applying Proposition~\ref{test} in a
nontrivial way, it must be assumed that
$S(x^* \circ \Phi,\varepsilon)
\cap B \ne \emptyset$. For example, in the
next corollary the requirement is satisfied, since $B=B_X$.
Note also that the requirement on $\Phi$ of being surjective from
$B_X$ to $B_X$ ensures that the slices $S(x^* \circ \Phi,\varepsilon)$
are not empty themselves.

\begin{corollary}
Let $\Phi,\Psi:B_X \to X$ be norm one functions.
If there exist $z \in S_X$ and $K>0$ such that for all $x_1,\dots ,x_n
\in B_X$ and $\alpha_1,\dots ,\alpha_n\ge0$ such that
$\sum_{i=1}^n \alpha_i =1$ there is an element $x \in B_X$ such that
the inequality
$$
\sum_{i=1}^n \alpha_i \|\Psi(x_i) - z\| \le
K \Bigl\|  x- \sum_{i=1}^n \alpha_i\Phi(x_i) \Bigr\|
$$
holds, then for each $\varepsilon>0$ there exist $\delta >0$ and $x^*_0
\in S_{X^*}$ such that $\Psi(S(x^*_0 \circ \Phi,\delta)) \subset
B_{K\varepsilon}(z)$.
\end{corollary}

\begin{proof}
Fix some  $x_1,\dots ,x_n \in X$ and $\alpha_1,\dots ,\alpha_n$ and
consider the element $x \in B_X$ given in the statement.
Using the  inequality we obtain
\bea
\sum_{i=1}^n \alpha_i \|\Psi(x_i) - z\|
&\le& K \sup_{x^* \in B_{X^*}}
\Bigl\langle x- \sum_{i=1}^n \alpha_i \Phi(x_i), x^* \Bigr\rangle  \\
&\le& K \sup_{x^* \in B_{X^*}} \Bigl( 1 - \Bigl\langle \sum_{i=1}^n
\alpha_i\Phi(x_i), x^* \Bigr\rangle\Bigr).
\eea
An application of Proposition~\ref{test} gives the result.
\end{proof}

\begin{example}
Take $X=C(K)$ for a perfect $K$, $\Phi(x)=x^2$ and $\Psi(x)=( \int_K x^2
\, d \mu ) y$ for a probability 
measure on $K$ and a fixed function $y \in
S_{C(K)}$. Then taking $z=y$ we get
\bea
\sum_{i=1}^n \alpha_i \Bigl\|\Bigl(\int_K x^2_i \,d \mu\Bigr)y - z\Bigr\|
&\le&
\sum_{i=1}^n \alpha_i \Bigl(1-\int_K x^2_i \,d \mu\Bigr)\|z\| \\
&=& \int_K d \mu - \sum_{i=1}^n \alpha_i \int_K x^2_i \,d \mu \\
&\le&  \Bigl\| \mathbf{1} - \sum_{i=1}^n \alpha_i x_i^2 \Bigr\|
\eea
for each finite set of functions $x_1,\dots ,x_n \in B_{C(K)}$ and $0
\le \alpha_1,\dots ,\alpha_n$ such that $\sum_{i=1}^n \alpha_i =1$.

Consequently, the result holds and for each $\varepsilon >0$ there is
a slice $S(x^*_0 \circ \Phi,\delta)$ such that $\Psi(S(x^*_0 \circ
\Phi,\delta)) \subset B_\varepsilon(z)$.
However, observe that the slices $S(x^*_0 \circ \Phi,\delta)$ can be
empty in this case, and so the Daugavet equation cannot be assured in
general by applying
 Remark~\ref{compo}(1). In fact, the
 equation does not hold if one takes for example $y= - \mathbf{1}$; in this
 case,
 $$
 \sup_{x \in B_{C(K)}} \Bigl\|x^2+ \Bigl(\int_K x^2 \, d\mu\Bigr)(-\mathbf{1})
 \Bigr\| \le 1.
 $$
 However, if we take $y=\mathbf{1}$, we obtain
 $\sup_{x \in B_{C(K)}} \|x^2+ (\int_K x^2 \, d \mu) \mathbf{1} \| =
 2$, and the  Daugavet equation holds.

  Note that
 Remark~\ref{compo}(1)  provides  the Daugavet equation for the
 ``order~$3$ version" of this result, since $\Phi(x)=x^3$ satisfies
 $\Phi(B_{C(K)})=B_{C(K)}$. Therefore, due to the Daugavet property of
 $C(K)$, for every $\mu \in S_{C(K)^*}$ and
 $y \in S_{C(K)}$ we have
 $$
 \sup_{x \in B_{C(K)}} 
\Bigl\| x^3 + \Bigl(\int_K x^3 \,d \mu\Bigr)y\Bigr\|=2.
 $$
\end{example}


\subsection{The case of $L^1(\mu)$-spaces for non-atomic measures $\mu$}

In this subsection we analyse several functions $\Phi$ that  are
natural candidates for being functions $\Phi$ on (the unit ball of)
$L^1$ in the results exposed in the previous sections.

Some cases that are in a sense canonical for applying our results are
the following. The first one given by the function $\Phi_0(f):= |f|$,
$f \in L^1(\mu)$.
The second case is the function $\Phi_\ast:=B_{L^1[0,1]} \to
B_{L^1[0,1]}$ given by the expression
$\Phi_\ast(f)= |f| \ast |f|$, where $\ast$ denotes the convolution in
$L^1[0,1]$; the third one is given by the formula $\Phi_2(f):= (\int_\Omega
|f| \,d \mu) \cdot f$. Adapting the proof of Theorem~2.6 and
Proposition~2.7 in \cite{polL1} 
that is based in some classical arguments for the Daugavet property
in $L^1(\mu)$, we obtain the following results, which can be applied
to these examples.

\begin{lemma} \label{la1}
Let $(\Omega,\Sigma,\mu)$ be a non-atomic measure space. Let
$\mathcal W$ be a set of norm one scalar functions in
$\ell_\infty (B_{L^1(\mu)})$. Let $\Phi:B_{L^1(\mu)} \to L^1(\mu)$ be
a norm one function 
such that $\|\Phi(z)\|=1$ for each $z \in S_{L^1(\mu)}$ and
satisfying also that for each $\delta, \varepsilon >0$ and $x' \in
\mathcal W$
we can find a norm one simple function $z$ such that $\mu(\supp
\Phi(z)) < \delta$ and
$|x'(\Phi(z))| > 1 - \varepsilon$.
Then
$$
\|\Phi + x'\otimes y\| = 2
$$
for all $x'\in \mathcal{W}$, $y\in S_{L^1(\mu)}$. 
\end{lemma}

\begin{proof}
We use Proposition~\ref{geom0}. Let $\varepsilon >0$, $x' \in \mathcal
W$ and $y \in S_{L^1(\mu)}$. Let us show that we can find $\omega$ and
an element $x \in S(\omega x',\varepsilon)$ such that
$$
\| \omega\Phi(x) +y \| > 2-2 \varepsilon.
$$
First note that there exists  $\delta >0$ such that $\int_A |y| \,d\mu <
\varepsilon$ for each $A \in \Sigma$ such that $\mu(A) < \delta$.
By the requirement on $\Phi$ for these $\delta >0$ and $\varepsilon
>0$ and  choosing an $\omega \in \mathbb T$ such
that $\omega x'(z)=|x'(z)|$, we have  that $ z \in S(\omega x',\varepsilon)$.
Thus we obtain
\bea
\|y+ \omega \Phi(z)\| &=&
\int_{\Omega \setminus \supp \Phi(z)} |y| \,d \mu  +
\int_{\supp \Phi(z)} |y + \omega \Phi(z)| \,d \mu \\
&\ge& \|y\| - \int_{\supp \Phi(z)} |y| \,d \mu +
\|\Phi(z)\| -  \int_{\supp \Phi(z)} |y| \,d \mu \\
&>& 2 - 2 \varepsilon.
\eea
Proposition~\ref{geom0} gives the result.
\end{proof}

\begin{lemma} \label{la1-2}
Let $(\Omega,\Sigma,\mu)$ be a  non-atomic measure space. Let
$\mathcal W$ be a set of norm one scalar functions from $L^1(\mu)$
that are weakly sequentially continuous. Let $\Phi:B_{L^1(\mu)} \to
L^1(\mu)$ be a norm one map that maps $S_{L^1(\mu)}$ onto $S_{L^1(\mu)}$. 
Then
$$
\|\Phi + x'\otimes y\| = 2
$$
for all $x'\in \mathcal{W}$, $y\in S_{L^1(\mu)}$. 
\end{lemma}

\begin{proof}
Let $x' \in \mathcal W$ and let $\delta, \varepsilon >0$. Since it is
weakly sequentially continuous, by Lemma 2.5 in \cite{polL1}, we can
find a norm one simple function $x$ such that $\mu(\supp x) < \delta$
and
$|x'(\Phi(x))| > 1 - \varepsilon$. The surjectivity of $\Phi$ provides an
element $z \in S_{L^1(\mu)}$ such that $\Phi(z)=x$. This $z$ satisfies
the requirement for $\Phi$ in Lemma~\ref{la1}; hence   the result holds.
\end{proof}

 In order to adapt  the results on weak sequential continuity that are
 shown to be useful in the case of the polynomial Daugavet property
 for $L^1(\mu)$ (see \cite{polL1}), there are two requirements on
 $\Phi$ that are useful and are included in the following
 definition.

In the next proposition,
we call a function $\Phi: B_{L^1(\mu)} \to L^1(\mu)$  \textit{admissible}
if the following requirements are satisfied.

(i) $\Phi$ must send functions of small support to functions of small
support, i.e.,  for each
$\delta >0$ there is a $\delta' >0$ such that for a function $f \in
L^1(\mu)$ with support satisfying $\mu(\supp f)<\delta'$, we have that
$\mu(\supp_{\Phi(f)}) <  \delta$.

(ii)  For all $f \in S_{L^1(\mu)}$, $\|\Phi(f)\|=1$.

Note that the mappings $\Phi_0, \Phi_*$ and $\Phi_2$ mentioned at the
beginning of this subsection are admissible.

\begin{proposition} \label{la2}
Let $(\Omega,\Sigma,\mu)$ be a non-atomic measure space. Let
$\Phi:B_{L^1(\mu)} \to L^1(\mu)$ be a norm one admissible function. Let
$\mathcal W \subset \ell_\infty(B_{L^1(\mu)})$  be a set of norm one
scalar functions from 
$B_{L^1(\mu)}$ to $\mathbb K$ such that $x' \circ \Phi$ is norm one and
weakly sequentially continuous for each $x' \in \mathcal W$. 
Then
$$
\|\Phi + x'\otimes y\| = 2
$$
for all $x'\in \mathcal{W}$, $y\in S_{L^1(\mu)}$. 
\end{proposition}

\begin{proof}
We use Lemma~\ref{la1}. Let $\varepsilon, \delta >0$ and $p \in \mathcal W$.
Note that since $\Phi$ is admissible,  there is a
$\delta' >0 $ such that if
$f\in L^1(\mu)$ and
$\mu(\supp f)<\delta'$, we have that $\mu(\supp {\Phi(f)}) <
\delta$.

 Since $x' \circ \Phi$ is weakly sequentially continuous, by Lemma 2.5
 in \cite{polL1}, we can find a norm one simple function $z$ such that
 $\mu(\supp z) < \delta'$ and
$|x'(\Phi(z))| > 1 - \varepsilon$.
Finally, notice that  we also have that $\mu(\supp \Phi(z)) < \delta$,
by the admissibility of $\Phi$.
Lemma~\ref{la1} gives the result.
\end{proof}


\end{document}